\documentclass[12pt, reqno, a4paper]{amsart}
\usepackage{amsmath,amsfonts,color, comment, epsfig,comment, amssymb,latexsym, amstext, enumerate, graphicx, mathtools}
\usepackage[noadjust]{cite}
\usepackage[unicode,plainpages=false,hyperindex,pdfpagemode=None,bookmarksopen,colorlinks,linkcolor=black,citecolor=blue,urlcolor=blue]{hyperref}

\newtheorem{thm}{Theorem}[section]

\newtheorem{prop}[thm]{Proposition}

\newtheorem{cor}[thm]{Corollary}

\newtheorem{lem}[thm]{Lemma}

\newtheorem*{thm*}{Theorem}

\theoremstyle{definition}
\newtheorem{defn}[thm]{Definition}

\theoremstyle{remark}

\newtheorem{exam}[thm]{Example}

\newcommand{\lip}{\mathrm{Lip}}

\newcommand{\Lpt}{L^\pt}

\newcommand{\loc}{\mathrm{loc}}
\newcommand{\pt}{\mathrm{pt}}
\newcommand{\Lloc}{L^\loc}

\numberwithin{equation}{section}

\begin{document}

\title{Weighted composition operators preserving various Lipschitz constants}
\author{Ching-Jou Liao, Chih-Neng Liu, Jung-Hui Liu \and Ngai-Ching Wong}

\address[C.-J. Liao]{Department of Mathematics, Hong Kong Baptist University, Hong Kong.}
\email[C.-J. Liao]{cjliao@hkbu.edu.hk}

\address[J.-H. Liu]{School of Information Engineering, Sanming University,
No.~25, Jing Dong Road San Yuan District, Sanming City, Fujian 365004, China.}
\email[J.-H. Liu]{18760259501@163.com}

\address[C.-N. Liu and N.-C. Wong]{Department of Applied Mathematics,
         National Sun Yat-sen University,
         Kaohsiung, 80424, Taiwan}
\email[C.-N. Liu]{bardiel0213@gmail.com}
\email[N.-C. Wong]{wong@math.nsysu.edu.tw}
\date{June 8, 2023}

\dedicatory{Dedicated to Professor Anthony To-Ming Lau on the occasion of his 80th birthday}


\begin{abstract}
Let $\lip(X)$, $\lip^b(X)$,
 $\lip^{\loc}(X)$  and $\lip^\pt(X)$ be the   vector spaces of  Lipschitz,  bounded Lipschitz,
 locally Lipschitz  and pointwise Lipschitz
 (real-valued) functions defined on a metric space $(X, d_X)$, respectively.
We show that if a weighted composition operator $Tf=h\cdot f\circ \varphi$ defines a bijection between such vector spaces
preserving Lipschitz constants, local Lipschitz constants or pointwise Lipschitz constants,
then $h= \pm1/\alpha$ is a constant function for some scalar $\alpha>0$ and $\varphi$ is an $\alpha$-dilation.

Let $U$ be  open connected and $V$ be open, or both $U,V$ are convex bodies, in normed linear spaces $E, F$, respectively.
Let $Tf=h\cdot f\circ\varphi$
be a bijective weighed composition operator between the vector spaces $\lip(U)$ and $\lip(V)$,
$\lip^b(U)$ and $\lip^b(V)$, $\lip^\loc(U)$ and $\lip^\loc(V)$,
or $\lip^\pt(U)$ and $\lip^\pt(V)$,
preserving the Lipschitz, locally Lipschitz, or pointwise Lipschitz constants,
respectively.  We show that there is a linear isometry $A: F\to E$, an $\alpha>0$ and a vector $b\in E$ such that
$\varphi(x)=\alpha Ax + b$, and
$h$ is a constant function assuming value $\pm 1/\alpha$.
More concrete results are obtained for the
special cases when $E=F=\mathbb{R}^n$, or when $U,V$ are $n$-dimensional flat manifolds.
\end{abstract}



\maketitle

\section{Introduction}

Let $f: (X,d_X)\to (Y,d_Y)$ be a map between metric spaces.  Let $B(p,\epsilon)$ be the
open ball in the metric space $X$ centered at $p$ of radius $\epsilon>0$.
We say that $f$ is  \emph{Lipschitz}
 if
the \emph{Lipschitz constant}
$$
L(f)=\sup_{x\neq y \in X}\frac{d_Y(f(x),f(y))}{d_X(x,y)}< +\infty;
$$
we say that $f$ is \emph{locally Lipschtiz} if the \emph{local Lipschtiz constants}
$$
\Lloc_{p}(f)=\lim_{\epsilon\to 0^+}\sup_{x\neq y \in B(p,\epsilon)}\frac{d_Y(f(x),f(y))}{d_X(x,y)}<+\infty, \quad\forall p\in X;
$$
and we say that  $f$ is \emph{pointwise Lipschitz} if  the
\emph{pointwise Lipschitz constants}
$$
\Lpt_p(f)=\limsup_{x \to p}\frac{d_Y(f(x),f(p))}{d_X(x,p)}<+\infty, \quad\forall p\in X.
$$
Let $\lip(X)$, $\lip^b(X)$, $\lip^{\loc}(X)$ and  $\lip^{\pt}(X)$ denote the (real) vector spaces of Lipschitz,
 bounded Lipschitz,
 locally Lipschtiz,
and  pointwise Lipschitz \emph{{}} (real-valued) functions on $X$ into $Y=\mathbb{R}$, respectively.
Clearly, $\lip^b(X)\subseteq \lip(X)\subseteq \lip^{\loc}(X) \subseteq \lip^{\pt}(X)$ and $0\leq \Lpt(f) \leq \Lloc\leq L(f)$,
in general, but all the inclusions and inequalities can be strict.
However, for a bounded metric space $X$ we have $\lip(X)=\lip^b(X)$.
On the other hand, $\lip(X)=\lip(\overline{X})$ where $\overline{X}$ is the metric completion of $X$.
See, e.g.,  \cite[Examples 2.6 and 2.7]{DC-J2010}, for more details.

 We are interested in the question how the (resp.\ local, pointwise) Lipschitz constants determine the (resp.\ local, pointwise)
Lipschitz function spaces.
For example, if $T:\lip(X)\to \lip(Y)$ is a bijective linear map preserving Lipschitz constants, namely,
$$
L(Tf)=L(f), \quad\text{for every $f\in \lip(X)$,}
$$
we ask if the underlying metric spaces $X,Y$ are equivalent, and if $T$ carries some good structure.
The following results give us motivations.

\begin{prop}\label{thm:T=is-wco}
\begin{enumerate}[(a)]
\item
\textup{(Weaver \cite[Theorem D]{Wea1995}; see also \cite[Sections 2.6 and 2.7]{Wea1999})}
Let $X,Y$ be complete metric spaces of diameter at most $2$, which   cannot be written as
a disjoint union of two nonempty subsets
with distance  $\geq 1$.
If $T:\lip(X)\rightarrow\lip(Y)$ is a surjective
linear isometry with respect to the norm $\|f\|=\max \{\|f\|_{\infty},L(f)\}$, then
  $$
  Tf=\alpha\cdot f\circ \varphi, \quad\forall f\in \lip(X),
  $$
  where $\varphi:Y\rightarrow X$ is a surjective isometry and
  $\alpha$ is a unimodular constant.

\item
\textup{(Wu  \cite{Wu2006}, Araujo and Dubarbie \cite[Corollary 6.1]{AraDub09})}
Let $X$ be a bounded complete metric space and $Y$ be a  compact metric space.
If $T:\lip(X)\rightarrow \lip(Y)$ is a  linear bijection preserving zero products, that is,
$$
fg=0 \quad\implies\quad TfTg =0,
$$
 then
$$
Tf=h\cdot f\circ \varphi, \quad\forall f\in \lip(X),
$$
where $\varphi$ is a bijective Lipschitz map from $Y$ onto $X$ with Lipschitz inverse $\varphi^{-1}$, and $h \in \lip(Y)$
is nonvanishing.
\end{enumerate}
\end{prop}
\noindent
See also, e.g., \cite{DC-J2010,Li2014, Wolf81}.
However, if  the map $T$ preserves only the Lipschitz constants without other assumption, $T$
 needs not carry a weighted composition operator form.
For example, let $\Psi$ be an arbitrary linear functional of the vector space $\lip(X)$.  Then the assignment
$Tf(x) = f(x) + \Psi(f)$ defines a linear bijective map from $\lip(X)$ onto $\lip(X)$
 preserving Lipschitz constants.  But $T$ does not assume a weighted composition form.

In this paper, based on \cite{Liu-thesis}, we study the structure of a weighted composition operator  $Tf=h\cdot f\circ \varphi$
 between
Lipschitz function spaces defined on metric spaces $X$ and $Y$, which preserves  Lipschitz constants.
We will see that $\varphi$ is an \emph{$\alpha$-dilation}
for a positive constant $\alpha$, i.e., $d_X(\varphi(u), \varphi(v)) = \alpha d_Y(u,v)$, $\forall u,v\in Y$,
 and $h$ is  a constant function assuming the value  $\pm 1/\alpha$.
 In particular, if   $U$ is open connected and $V$ is open, or both $U,V$ are
convex bodies,
in normed linear spaces
$E,F$, respectively, then
 $\varphi(x) = \alpha Ax + b$ for a linear isometry $A: F\to E$ and a vector $b\in E$.  In other words,
 $$
 Tf(x)=\pm\alpha^{-1}f(Ax+b).
 $$
When $E=F=\mathbb{R}^n$, we obtain similar results for those bijective weighted composition operators
preserving  local or pointwise Lipschitz constants.
We also extend these results to the case when $U, V$ are flat manifolds.

\section{Weighted composition operators  preserving Lipschitz constants}

All vector spaces discussed in this paper are over the real number field $\mathbb{R}$.

In this section, we study Lipschitz constant preserving weighted composition operators between Lipschtiz function spaces.
We begin with some well known results.
Here, by an \emph{$\alpha$-dilation} (resp.\ \emph{isometry}) between metric spaces $X, Y$ for $\alpha>0$,
we mean a map $\varphi: X\to Y$ such that
$d_Y(\varphi(x), \varphi(y))= \alpha d_X(x, y)$ (resp.\ when $\alpha=1$) for all $x,y\in X$.  Moreover, a
\emph{convex body} in a normed linear space is a (not necessarily closed or bounded) convex set with nonempty interior.

\begin{prop}\label{Man1972}
\begin{enumerate}[(a)]
  \item \textup{({Mazur-Ulam theorem; see, e.g., \cite{Fle2003}})}
  Every surjective isometry $T: M\to N$ between normed linear spaces is affine; namely,
  $$
  T(\lambda x + (1-\lambda) y)= \lambda Tx + (1-\lambda)Ty, \quad\text{for all $x,y\in M$ and  $0\leq \lambda \leq 1$.}
  $$
  \item \textup{(Mankiewicz \cite{Man1972})}
Every isometry $\varphi: U\to V$  from an open connected
set (resp.\   convex body) $U$ in a normed linear space $M$
onto an open set (resp.\ convex body) $V$ in a normed linear space $N$
has a unique isometric extension  from $M$ onto $N$.
\end{enumerate}
\end{prop}

Recall that a metric space $(X, d_X)$ is \emph{quasi-convex} if  there is
a constant $C >0$ such that for any $x, y\in X$ there is a continuous
path in $X$ joining $x$ to $y$ with
length at most $Cd_X(x,y)$.

\begin{prop}[{\cite[Theorems 3.9 and 3.12]{Gar2004}}]\label{prop:Gar2004}
Let $\varphi:Y\to X$ be a map between   metric spaces.  Then $\varphi$ is Lipschitz if and only if
\begin{enumerate}[(a)]
\item  $f\circ \varphi\in \lip(Y)$ for each $f\in \lip(X)$; or
\item   $f\circ \varphi\in \lip^b(Y)$ for each $f\in \lip^b(X)$, provided that  both $X, Y$ are quasi-convex.
\end{enumerate}
\end{prop}

\begin{thm}\label{thm:wecomop}
Let $(X,d_X)$ and  $(Y,d_Y)$ be (resp.\ quasi-convex) metric spaces.
Let $T:\lip(X)\to \lip(Y)$ (resp.\ $T: \lip^b(X)\to\lip^b(Y)$) be a bijective weighted composition operator
$Tf=h\cdot f\circ \varphi$
preserving Lipschitz constants.
Then $\varphi$ is an $\alpha$-dilation from $Y$ onto a dense subset $X_0=\varphi(Y)$ of $X$ for  $\alpha>0$,
and $h$ is a constant function assuming either $1/\alpha$ or $-1/\alpha$.
If $Y$ is complete, then $\varphi(Y)=X$.
\end{thm}
\begin{proof}
We verify the case when $T$ sends $\lip(X)$ onto $\lip(Y)$.  The other case follows similarly.

Since $T$ preserves Lipschitz constants, so does its inverse; indeed,
$$
L(g)=L(TT^{-1}g)=L(T^{-1}g), \quad\text{for all $g$ in $\lip(Y)$.}
$$
Observe that
$$
0=L(1)=L(T1)= L(h) =\sup_{u\neq v}\frac{|h(u)-h(v)|}{d_Y(u,v)}.
$$
Hence $h= \pm1/\alpha$ is a nonzero constant function for a scalar $\alpha>0$.

For any $f$ in $\lip(X)$,
we have $f\circ \varphi=\pm\alpha  Tf\in\lip(Y)$.
It follows from Proposition \ref{prop:Gar2004} that $\varphi$ is Lipschitz.
Since $T$ is bijective, $\varphi$ is one-to-one with dense range.
 Let $X_0=\varphi(Y)$, which is a dense subset of the metric space $X$.
Define a bijective linear map $S: \lip(Y)\to \lip(X_0)$ such that $Sg=T^{-1}g\mid_{X_0}$;
in other words,
$Sg = \pm\alpha  g\circ\psi$, where $\psi$ is the
 bijection from $X_0$ onto $Y$ defined by the condition $\psi(x)=y$ whenever $x=\varphi(y)$.
 Then a similar argument shows that $\psi$ is  Lipschitz from  $X_0$ onto $Y$.
Let $\overline{\varphi}: \overline{Y}\to \overline{X}$ and
$\overline{\psi}: \overline{X}\to \overline{Y}$ be the unique Lipschitz extensions of $\varphi$ and $\psi$
between the metric completions $\overline{X}, \overline{Y}$ of $X,Y$, respectively.
For any $x\in \overline{X}$, let $x_n\in X_0$ such that $x_n\to x$, we have
$\overline{\varphi}(\overline{\psi}(x))= \lim_n \overline{\varphi}(\overline{\psi}(x_n))
= \lim_n \varphi(\psi(x_n))= \lim_n x_n = x$.  It amounts to saying that $\overline{\varphi}\circ\overline{\psi}= I_{\overline{X}}$.
In a similar manner, we see that $\overline{\psi}\circ\overline{\varphi}= I_{\overline{Y}}$, and thus
$\overline{\varphi}^{-1}=\overline{\psi}$ is also Lipschitz.

On the other hand, the fact
$$
L(f)=L(Tf)=L(\pm\alpha^{-1} f\circ\varphi) \leq \alpha^{-1}L(f)L(\varphi), \quad\forall f\in \lip(X),
$$
implies that $L(\varphi) \geq  \alpha$.
Assume $L(\varphi)> \alpha$. Then there are $p$ and $q$ in $Y$, such that
$$
d_X(\varphi(p),\varphi(q)) >  \alpha d_Y(p,q).
$$
Define $\tilde{f}:X\to \mathbb{R}$ by
$$
\tilde{f}(x)=\min \{d_X(x,\varphi(p)), d_X(\varphi(p),\varphi(q))\}.
$$
Then
\begin{align*}
|\tilde{f}(x)-\tilde{f}(y)|&\leq |d(x,\varphi(p)) - d(y,\varphi(p))|\\
                      &\leq d_X(x,y), \quad\text{for all $x$, $y$ in $X$,}
\end{align*}
and
$$
\|\tilde{f}(\varphi(q))-\tilde{f}(\varphi(p))\|=d(\varphi(q),\varphi(p)).
$$
Hence, $\tilde{f}\in \lip^b(X)$ with $L(\tilde{f})=1$.
But
\begin{align*}
|T\tilde{f}(p)-T\tilde{f}(q)|
=\, &  |h(p)\tilde{f}\circ\varphi(p)-h(q)\tilde{f}\circ\varphi(q) |\\
=\, &  \alpha^{-1}  d_X(\varphi(p),\varphi(q))\\
>\, & d_Y(p,q).
\end{align*}
This implies $L(T\tilde{f})>1$, which is a contradiction. Hence $L(\varphi)= \alpha$. Similarly, $L(\psi)=\alpha^{-1}$.

If $u\neq v$ in $Y$ and $\varphi(u)=s$, $\varphi(v)=t$ in $X$, then
$$
\alpha=L(\varphi)\geq\frac{d_X(\varphi(u),\varphi(v))}{d_Y(u,v)}
=\frac{d_X(s,t)}{d_Y(\psi(s),\psi(t))}\geq \frac{1}{L(\psi)}= \alpha,
$$
and thus
$$
d(\varphi(u),\varphi(v))=\alpha d_Y(u,v).
$$
In other words, $\varphi$ is an $\alpha$-dilation.

Finally, if $Y=\overline{Y}$ is complete then $X_0=\varphi(Y)$ is also complete, and thus $\varphi(Y)=\overline{X_0}=X$.
\end{proof}

\begin{cor}\label{cor:wcL-cv}
  Assume either that  $U,V$ are open sets and $U$ is connected, or that
   both $U,V$ are convex bodies, in normed linear spaces $E, F$, respectively.
  Let the weighed composition operator $Tf=h\cdot f\circ\varphi$ define a bijection
  from $\lip(U)$ onto $\lip(V)$ preserving Lipschitz constants.
  Then there is a  linear isometry $A:  F\to E$ with dense range, a scalar $\alpha>0$ and a vector $b\in E$ such that
  $\varphi(y) = \alpha Ay + b$
  for all $y\in V$, and $h$ is a constant function assuming value $\pm \alpha^{-1}$.
  In other words,
  $$
  Tf(y) = \pm \alpha^{-1} f(\alpha Ay + b), \quad\text{for all $f\in \lip(U)$ and for all $y\in V$.}
  $$
\end{cor}
\begin{proof}
By Theorem \ref{thm:wecomop}, we know that $h=\pm\alpha^{-1}$ is a constant function for some $\alpha>0$, and
$\varphi$ is an $\alpha$-dilation from $V$ onto a  dense subset of $U$.
Indeed, $\varphi$ extends uniquely to an $\alpha$-dilation $\overline{\varphi}$ from $\overline{V}$ onto
$\overline{U}$, where $\overline{U},\overline{V}$ are the closures
 of $U,V$ in the Banach space completions $\overline{E}, \overline{F}$ of $E,F$, respectively.

If $V$ is connected in $F$, then so is the interior of $\overline{V}$ in $\overline{F}$.
It is plain that $\overline{\varphi}$ sends the interior of $\overline{V}$ in the Banach space $\overline{F}$
onto the interior of $\overline{U}$ in the Banach space $\overline{E}$.
Then $\overline{\varphi}$ extends uniquely to an affine $\alpha$-dilation from $\overline{F}$
onto $\overline{E}$ by Proposition \ref{Man1972}.
In the other case,  $U, V$ are both convex bodies in $E,F$,
and thus so are $\overline{U},\overline{V}$ in the Banach spaces  $\overline{E}, \overline{F}$, respectively.
 It follows from Proposition \ref{Man1972} again that $\overline{\varphi}$ extends uniquely to an
 affine $\alpha$-dilation from $\overline{F}$ onto $\overline{E}$.
  In both cases, $\varphi$ extends to a unique
 affine $\alpha$-dilation from $F$ into $E$ with dense range.  The assertion follows.
\end{proof}

\begin{cor}\label{cor:wcL-cv-bdd}
  Let $U,V$ be convex bodies in normed linear spaces $E, F$, respectively.  Let
  the weighed composition operator $Tf=h\cdot f\circ\varphi$ define a bijection
  from $\lip^b(U)$ onto $\lip^b(V)$ preserving Lipschitz constants.
  Then there is a surjective linear isometry $A:  F\to E$, a scalar $\alpha>0$ and a vector $b\in E$ such that
  $\varphi(y) = \alpha Ay + b$
  for all $y\in V$, and $h$ is a constant function assuming value $\pm \alpha^{-1}$.
  In other words,
  $$
  Tf(y) = \pm \alpha^{-1} f(\alpha Ay + b), \quad\text{for all $f\in \lip(U)$ and for all $y\in V$.}
  $$
\end{cor}
\begin{proof}
  Note that convex bodies in normed spaces are quasi-convex.  The proof now goes the same way as in Corollary \ref{cor:wcL-cv}.
\end{proof}

\begin{exam}\label{exam:wecomopRn}\label{thm:wecomop[0,1]}\label{thm:wecomop[a,b]}
(a) Let $Tf=h\cdot f\circ \varphi$ define a bijection $T:\lip(\mathbb{R}^n)\to \lip(\mathbb{R}^n)$
preserving Lipschitz constants.
  It follows from Theorem \ref{thm:wecomop} that $\varphi$ is Lipschitz, and indeed an $L(\varphi)$-dilation of $\mathbb{R}^n$.
Moreover,  $h$ is  the constant function assuming either $1/L(\varphi)$ or $-1/L(\varphi)$.
By Proposition \ref{Man1972},  there is an $n\times n$ orthogonal matrix $A$ such that
$\varphi(x)=L(\varphi)Ax+\varphi(0)$.  Let $\alpha=L(\varphi)$ and $b=\varphi(0)$, we have
$$
Tf(x) = \pm \alpha^{-1}f(\alpha Ax+b) \quad\text{for all $x\in \mathbb{R}^n$.}
$$

(b)
Let $Tf=h\cdot f\circ \varphi$ define a bijection  $T:\lip([0,1])\to \lip([0,1])$  preserving Lipschitz constants.
Applying Theorem  \ref{thm:wecomop}, we see that $\varphi$ is an $L(\varphi)$-dilation of $[0,1]$.
This forces $L(\varphi)=1$, and  either
$\varphi(x)=x$  or $\varphi(x)=1-x$.
Consequently, $T$ is given by
$$
Tf(x) = \pm f(x)\quad\text{or}\quad Tf(x)= \pm f(1-x).
$$

(c)  More generally,
let $Tf=h\cdot f\circ \varphi$ define a bijection $T:\lip([a,b])\to \lip([c,d])$  preserving Lipschitz constants.
Then $L(\varphi)=\frac{b-a}{d-c}$,
$h=\pm \frac{d-c}{b-a}$, and either
$$
\varphi(x)=\frac{b-a}{d-c}(x-c)+a\quad\text{or}\quad \varphi(x)=\frac{b-a}{d-c}(d-x)+a.
$$
Consequently,
$$
T(f) = \pm \frac{d-c}{b-a} f\left(\frac{b-a}{d-c}(x-c)+a\right), \quad\text{or}\quad
T(f) = \pm \frac{d-c}{b-a} f\left(\frac{b-a}{d-c}(d-x)+a\right).
$$

(d)
Let $T: \lip([0,1]^n)\to \lip([0,1]^n)$ be a bijective weighted composition operator preserving Lipschitz constants.
It follows from Corollary \ref{cor:wcL-cv} and the structure of the symmetric group of the $n$-cube that
$$
Tf(x) = \pm f(Px+b) \quad\text{for all $x\in [0,1]^n$,}
$$
for an $n\times n$  signed permutation matrix $P$ (in
the sense that exactly one entry in each row and each column is $\pm 1$, and  elsewhere $0$),
 and a vector $b\in \mathbb{R}^n$ in which all entries are either $0$ or $1$.

(e) There are similar versions for bounded Lipschitz functions in all above examples. The conclusions are identical.
\end{exam}

\section{Weighted composition operators preserving  local/pointwise Lipschitz constants}

We say that a weighted composition operator $Tf=h\cdot f\circ\varphi$ from  $\lip^\loc(X)$ into $\lip^\loc(Y)$,
or from $\lip^\pt(X)$ into $\lip^\pt(Y)$, preserves
the local or pointwise Lipschitz constants (with respect to $\varphi$), if
$$
\Lloc_p(Tf)=\Lloc_{\varphi(p)}(f)\quad \text{for all $f\in \lip(X)$ and all $p\in Y$,}
$$
or
$$
\Lpt_p(Tf)=\Lpt_{\varphi(p)}(f)\quad \text{for all $f\in \lip(X)$ and all $p\in Y$.}
$$

In this section we consider  weighted composition operators
between locally/pointwise Lipschitz functions on Euclidean spaces preserving local/pointwise Lipschitz constants.

\begin{prop}\label{prop:lip-loc-pt-comp}
Let $(X,d_X)$ and $(Y,d_Y)$ be metric spaces without isolated point, and let $\varphi:Y\to X$ be a map.
Suppose that
$$
f\circ \varphi\in \lip^\loc(Y) \ \text{{\rm(}resp.\ $\lip^\pt(Y)${\rm)}
 \quad for all $f\in \lip^\loc(X)$  {\rm(}resp.\ $\lip^\pt(X)${\rm)}.}
$$
 Then $\varphi$ is locally (resp.\ pointwise) Lipschitz from $Y$ into $X$.
 If the composition map $f\mapsto f\circ\varphi$ preserves local (resp.\ pointwise) Lipschitz constants
 then $\Lloc_q(\varphi)=1$ (resp.\ $\Lpt_q(\varphi)=1$) for every $q\in Y$.
\end{prop}
 \begin{proof}
Let $q\in Y$, and
   let $f_q(x)=\min \{d_X(x, \varphi(q)), 1\}$.  Then $f_q$ is   bounded and pointwise Lipschitz on $X$.
   In particular, $\Lpt_{\varphi(q)}(f_q) = 1$.
    Since $q$ is not an isolated point in $Y$, we see that $d_X(\varphi(y),\varphi(q))<1$ eventually when $y\to q$. Hence,
\begin{align*}
   \Lpt_q(f_q\circ \varphi) &= \limsup_{y\to q} \frac{|f_q(\varphi(y))-f_q(\varphi(q))|}{d_Y(y,q)}\\
   &=  \limsup_{y\to q} \frac{\min\{d_X(\varphi(y), \varphi(q)),1\}}{d_Y(y,q)} <+\infty.
\end{align*}
   Thus,
   $$
   \Lpt_q(\varphi)= \limsup_{y\to q} \frac{d_X(\varphi(y), \varphi(q))}{d_Y(y,q)}= \Lpt_q(f_q\circ \varphi)<+\infty.
   $$
Since this holds for every point $q$ in $Y$, we see that $\varphi$ is pointwise Lipschitz from $Y$ into $X$.
Moreover, $\Lpt_q(\varphi)=1$ for all $q\in Y$ if the composition map
$f\mapsto f\circ\varphi$ preserves pointwise Lipschitz constants.

The case  for local Lipschitzness is proved in \cite[Lemma 3.15]{Gar2004}.  As an alternative proof, consider
\begin{align*}
\Lloc_{\varphi(q)}(f_q) &= \limsup_{y,z\to \varphi(q)} \frac{|f_q(y)-f_q(z)|}{d_X(y,z)}\\
&\leq \limsup_{y,z\to \varphi(q)} \frac{|d_X(y,\varphi(q))-d_X(z,\varphi(q))|}{d_X(y,z)}\leq 1,
\end{align*}
while, in general,
$$
\Lloc_{\varphi(q)}(f_q) \geq \Lpt_{\varphi(q)}(f_q) = 1.
$$
Thus, $\Lloc_{\varphi(q)}(f_q)=1$.  In a similar fashion, we can verify that $\varphi$ is locally Lipschitz,
and $\Lloc_q(\varphi)=1$ when the composition map $f\mapsto f\circ\varphi$ preserves local Lipschitz constants.
 \end{proof}

\begin{lem}\label{lem:loca[a,b]}
Let $\varphi:[a,b]\to \mathbb{R}$ be a locally (resp.\ pointwise) Lipschitz function satisfying that
$$
\Lloc_p(\varphi) =\alpha \quad \text{(resp.\ $\Lpt_p(\varphi)=\alpha$)}, \quad \text{for all $p\in[a,b]$.}
$$
If $\alpha=0$ then $\varphi(x)=c$ for some fixed scalar $c$.
In general, if $\varphi$ is injective
then $\varphi(x)=\alpha x+c$ or $\varphi(x)=-\alpha x+c$ on $[a,b]$ for some fixed scalar $c$.
\end{lem}

\begin{proof}
Suppose
$$
L_p^\loc(\varphi)=
\limsup_{\substack{x,y \to p\\ x,y\in [a,b]}}\frac{|\varphi(x)-\varphi(y)|}{|x-y|}=\alpha\quad\text{for all $p\in[a,b]$.}
$$
Let $\epsilon>0$. For any point $p\in[a,b]$,
there is an open interval $B(p,\delta_p)$ centered at $p$ with radius  $\delta_p>0$ such that
$$
\frac{|\varphi(x)-\varphi(y)|}{|x-y|}<\alpha+\epsilon \quad\text{for all $x, y\in B(p,\delta_p)\cap [a,b]$.}
$$
By the Lebesgue's number lemma, there is a $\delta>0$ such that
for any partition $[a,b]=\bigcup_{i=1}^m [x_{i-1},x_i]$ with all $|x_i - x_{i-1}|< \delta$, we have
$x_{i-1}, x_i\in B(p_i, \delta_{p_i})$ for some $p_i\in [a,b]$, and thus
\begin{align}\label{eq:lloc-bv}
 \sum_{i=1}^{m}|\varphi(x_{i})-\varphi(x_{i-1})|
              < & \sum_{i=1}^{m}(\alpha+\epsilon)|x_{i}-x_{i-1}| \\
              = & (\alpha+\epsilon)(b-a)<+\infty. \notag
\end{align}
In particular, $\varphi$ is of bounded variation on $[a,b]$.
Consequently, $\varphi$ is differentiable almost everywhere in $[a,b]$,
and $|\varphi'(p)|=\alpha$ for almost all $p\in[a,b]$.
If $\alpha=0$ then $\varphi$ assumes constant value $c$ on $[a,b]$.

Suppose $\varphi$ is injective now.
Since $\varphi$ is continuous and injective, it is monotone on $[a,b]$.  Thus,
$\varphi'=\alpha$ or $\varphi'=-\alpha$   almost everywhere on $[a,b]$.
Hence, in the sense of Lebesgue integral,
\begin{align*}
\varphi(x)-\varphi(a)=\int_{a}^{x}\varphi'(t)\,dt=\pm\alpha(x-a).
\end{align*}
Consequently,
$\varphi(x)=\alpha x+c$ or $\varphi(x)=-\alpha x+c$ on $[a,b]$, where $c = \mp \alpha a + \varphi(a)$.

The other case when
$$
\Lpt_p(\varphi)=\limsup_{x \to p}\frac{|\varphi(x)-\varphi(p)|}{|x-p|}=\alpha, \quad  \text{for all $p\in[a,b]$,}
$$
follows
 similarly, except that in \eqref{eq:lloc-bv} we might break $|\varphi(x_i)-\varphi(x_{i-1})| \leq
|\varphi(x_i)-\varphi(p_i)| + |\varphi(p_i)-\varphi(x_{i-1})|$ where $x_{i-1}, x_i\in B(p_i; \delta_{p_i})$.
\end{proof}

The following example tells us that the injectivity condition  is indispensable.

\begin{exam}
Define $\varphi:[0,1]\to [\frac{1}{2},1]$ by
$$\varphi(x)=\left\{\begin{array}{ll}
        x,   & \frac{1}{2}\leq x\leq1,\\
        1-x, & 0\leq x\leq\frac{1}{2}.
        \end{array}\right.$$
Then $L^\loc_x(\varphi)=1$ for all $x\in[0,1]$.
But $\varphi$ is not linear.
\end{exam}

\begin{lem}\label{lem:wf-const}
  Let $V$ be  either an open  connected  set or a convex body in a normed linear space.
  Let $h$ be locally or pointwise Lipschitz function on $V$ with
 $\Lloc_q(h) =0$ or $\Lpt_q(h)=0$ for all $q\in V$.  Then
  $h$ is a constant function on $V$.
\end{lem}
\begin{proof}
Suppose $h\in \lip^\pt(V)$ with $\Lpt_q(h)=0$ for all $q\in V$.
  Let $(1-t)x + ty$ for  $t\in [0,1]$ be a path lying in the interior of  $V$.
  Define $f:[0,1]\to \mathbb{R}$ by $f(t) = h((1-t)x + ty)$.
  Then $f\in \lip^\pt([0,1])$ with $\Lpt(f)(t)=0$ for all $t\in [0,1]$.
  By Lemma \ref{lem:loca[a,b]}, $f$ is a constant function.  In other
  words, $h(x)=h(y)$.
   A connected argument shows that $h$ is constant on $V$ if $V$ is open and connected.
  If $V$ is a convex body, then we can see that $h$ is constant on  the interior of  $V$.  By continuity, $h$ is
  constant on $V$.
\end{proof}

We also have the following well known results.

\begin{prop}[Rademacher/Stepanov Theorem; see, e.g., \cite{Federer69, Maly99}]\label{prop:RademacherStepanov}\label{lem:loca}
 Let $V$ be an open subset of $\mathbb{R}^n$.  Every locally/pointwise Lipschitz function
 $\varphi(x)=(\varphi_1(x),\varphi_2(x),\cdots,\varphi_n(x))$
 from $V$ into $\mathbb{R}^n$
 is differentiable on $V$ almost everywhere.  In particular,
 for  every $v\neq 0$ in $\mathbb{R}$, the directional derivative
$D_v\varphi_{i}(x)=v\cdot\nabla\varphi_i(x)$ exists for almost every $x$  in $V$
for $i=1,2,\cdots,n$.
 \end{prop}

\begin{thm}\label{thm:locaU}\label{thm:loca[0,1]n}
Let $U$ and $V$ be two open connected sets or two convex bodies in $\mathbb{R}^n$.
Let $T:\lip^{\loc}(U)\to \lip^{\loc}(V)$
(resp.\ $T:\lip^\pt(U)\to \lip^\pt(V)$) be a bijective weighted composition operator defined by $Tf= h\cdot f\circ \varphi$
 preserving local (resp.\ pointwise) Lipschitz constants.
Then $h$ assumes constant  value $\pm \alpha^{-1}$ on $V$ for some $\alpha>0$, and
$\varphi$ assumes the form
$\varphi(v)=\alpha Av +b$ for an $n\times n$ orthogonal matrix $A$ and a point $b\in \mathbb{R}^n$.
\end{thm}
\begin{proof}
  We present the proof for the case when $U,V$ are open connected sets in $\mathbb{R}^2$,
   and $T:\lip^\loc(U)\to \lip^\loc(V)$ preserves local
  Lipschitz constants. The other cases follow similarly.

It is clear that $\varphi$ is injective and $\varphi(V)$ is a dense subset of $U$.
Since $h=T\mathbf{1}_U$ is locally Lipschitz with $L^\loc_{\varphi(p)}(h) = L^\loc_p(\mathbf{1}_U)=0$ for all $p\in V$,
we see that $h=\pm\alpha^{-1}$ is a constant function for some $\alpha>0$ by Lemma \ref{lem:wf-const}.
It  then follows from Proposition \ref{prop:lip-loc-pt-comp}  that both $\varphi$ and $\varphi^{-1}$ are locally Lipschitz.

We write $\varphi(x,y)=(\varphi_1(x,y),\varphi_2(x,y))$.
It follows from Proposition \ref{prop:RademacherStepanov} that
all partial derivatives $\dfrac{\partial\varphi_1}{\partial x}$, $\dfrac{\partial\varphi_1}{\partial y}$,
$\dfrac{\partial\varphi_2}{\partial x}$ and $\dfrac{\partial\varphi_2}{\partial y}$ exist,
and $\Lloc_{(x,y)}(\varphi_i)=\|\nabla \varphi_i (x,y)\|$ for $i=1,2$, for almost every point $(x,y)$ in $V$.
Let $I_1(x,y)=x$ and $I_2(x,y)=y$.
Then $I_1$ and $I_2$ belong to $\lip^{\loc}(V)$ with
$\Lloc_{(x,y)}(I_1)=\Lloc_{(x,y)}(I_2)=1$.
It follows  that
 $$
 \alpha^{-1}\Lloc_{\varphi(x,y)}(\varphi_1) =\Lloc_{\varphi(x,y)}(h\cdot I_1\circ \varphi) =\Lloc_{(x,y)}(I_1)=1
 $$
 for   almost every point $(x,y)$ in $V$.  Thus
$$
(\frac{\partial\varphi_1}{\partial x}(x,y))^2+(\frac{\partial\varphi_1}{\partial y}(x,y))^2=\alpha^2.
$$
Dealing with $I_2$ instead, we also have
$$
(\frac{\partial\varphi_2}{\partial x}(x,y))^2+(\frac{\partial\varphi_2}{\partial y}(x,y))^2=\alpha^2
$$
for almost all $(x,y)$ in $V$.

In general, for any $f$ in $\lip^{\loc}(U)$, by Proposition \ref{prop:RademacherStepanov} we have
$$
\Lloc_{(u,v)}(f) = \sqrt{\left(\frac{\partial f}{\partial x}\right)^2 + \left(\frac{\partial f}{\partial y}\right)^2}_{\mid(u,v)}
$$
for almost every $(u,v)$ in $U$.
Now let $f(x,y)=xy$  in $\lip^{\loc}(U)$.   We have
$\Lloc_{(u,v)}(f) = \sqrt{v^2 + u^2}$ for all $(u,v)$ in $U$.  In particular,
$$
\Lloc_{\varphi(x,y)}(f)=\sqrt{\varphi_1(x,y)^2+\varphi_2(x,y)^2} \quad\text{for  all $(x,y)$ in $U$.}
$$
Hence
\begin{align*}
\alpha \sqrt{\varphi_1^2+\varphi_2^2}\mid_{(x,y)}
&= \alpha \Lloc_{\varphi(x,y)}(f)\\
&= \alpha \Lloc_{(x,y)}(Tf) = \Lloc_{(x,y)}(\varphi_1(x,y)\varphi_2(x,y))\\
&= \sqrt{(\frac{\partial\varphi_1}{\partial x}\varphi_2+\frac{\partial\varphi_2}{\partial x}\varphi_1)^2
+(\frac{\partial\varphi_1}{\partial y}\varphi_2+\frac{\partial\varphi_2}{\partial y}\varphi_1)^2}\mid_{(x,y)}\\
&=\sqrt{\alpha^2(\varphi_1^2+\varphi_2^2)
+2(\frac{\partial\varphi_1}{\partial x}\frac{\partial\varphi_2}{\partial x}
+\frac{\partial\varphi_1}{\partial y}\frac{\partial\varphi_2}{\partial y})\varphi_1\varphi_2}\mid_{(x,y)}\\
\end{align*}
for almost all $(x,y)$ in $V$.
This implies
$$\frac{\partial\varphi_1}{\partial x}(x,y)\frac{\partial\varphi_2}{\partial x}(x,y)
+\frac{\partial\varphi_1}{\partial y}(x,y)\frac{\partial\varphi_2}{\partial y}(x,y)=0$$
for almost all $(x,y)$ in $V$.
Therefore,
$$
D[\varphi]_{\mid(x,y)} = \left(
                           \begin{array}{cc}
                             \dfrac{\partial\varphi_1}{\partial x} & \dfrac{\partial\varphi_1}{\partial y} \\
                             \dfrac{\partial\varphi_2}{\partial x} & \dfrac{\partial\varphi_2}{\partial y} \\
                           \end{array}
                         \right)_{\mid(x,y)}
$$
exists, and $\alpha^{-1} D[\varphi]_{\mid(x,y)}$ is an orthogonal matrix  for almost all $(x,y)$ in $V$.

For any two points $(x_1,y_1)$ and $(x_2,y_2)$ in an open ball $B$ contained in $V$, let
$C: [0,1]\to V$ be any smooth curve in $B$ joining $(x_1,y_1)$ to $(x_2,y_2)$.
Using Lebesgue integration, we have
\begin{align*}
\|\varphi(x_1,y_1)-\varphi(x_2,y_2)\| &\leq \int_0^1 \|\frac{d}{dt}(\varphi(C(t)))\|\,dt\\
&=\int_0^1 \|D[\varphi(C(t))]\frac{d}{dt}C(t)\|\,dt = \alpha \int_0^1 \|\frac{d}{dt}C(t)\|\,dt.
\end{align*}
Hence, taking infimum over all such smooth curves $C$, we have
$$
\|\varphi(x_1,y_1)-\varphi(x_2,y_2)\| \leq \alpha \|(x_1,y_1)-(x_2,y_2)\|.
$$
In particular, the injective map $\varphi$ is continuous on the open set $V\subset \mathbb{R}^n$.
By the invariance of domain,  $U_0=\varphi(V)$  is an open dense subset of the open set $U\subset\mathbb{R}^n$.
It is clear that $T^{-1}g\mid_{U_0}= \alpha g\circ\varphi^{-1}$ preserves local Lipschitz constants.
By a similar argument, we have $\alpha D[\varphi^{-1}]_{\mid q}$ exists as an orthogonal matrix for almost every
point $q\in U_0$.
Hence, the reverse inequality
$$
\|\varphi(x_1,y_1)-\varphi(x_2,y_2)\| \geq \alpha \|(x_1,y_1)-(x_2,y_2)\|
$$
also holds.
In other words,  $\varphi$ is an $\alpha$-dilation  from $B$ onto
$\varphi(B)$ for any open ball $B$ contained in $V$.
It   follows from Proposition \ref{Man1972} that $\varphi\!\mid_B$ can be uniquely extended to an $\alpha$-dilation
from $\mathbb{R}^n$ onto itself.  By a connectedness argument, we see that $\varphi$ can be extended uniquely to the same $\alpha$-dilation
of $\mathbb{R}^n$ onto itself.
In other words, $\varphi(v)=\alpha Av + b$ for an $n\times n$ orthogonal matrix $A$ and a vector $b\in \mathbb{R}^n$, as asserted.
\end{proof}

\begin{exam}\label{lem:loca[0,1]2}\label{thm:wloca[0,1]n}
 Let $T:\lip^{\loc}([0,1]^n)\to \lip^{\loc}([0,1]^n)$ be a bijective
 weighted composition operator 
preserving local Lipschitz constants.  Then
$$
Tf = \pm f(Px+b),
$$
 where $P$ is an   $n\times n$ signed permutation matrix and $b\in \mathbb{R}^n$ has entries either $0$ or $1$.
\end{exam}

\begin{exam}
Let $T:\lip^{\loc}([a,b])\to \lip^{\loc}([c,d])$ defined by $Tf=h\cdot f\circ \varphi$
be a bijection and preserve local Lipschitz constants. Then
$h=\frac{d-c}{b-a}$ or $h=-\frac{d-c}{b-a}$
and $\varphi(x)=\frac{b-a}{d-c}(x-c)+a$ or $\varphi(x)=\frac{b-a}{d-c}(d-x)+a$.
\end{exam}

\section{Local/pointwise Lipschitz constant preservers of flat manifolds}

In this section we consider locally (resp.\ pointwise) Lipschitz functions defined on flat manifolds.
A \emph{flat manifold} of dimension $n$ is a set $M$
with a family of injective mappings, called \emph{charts},
 $\phi_{\alpha}:U_{\alpha}\subseteq \mathbb{R}^n\to \phi_{\alpha}(U_{\alpha})\subseteq M$
of open connected sets $U_{\alpha}$ containing $0$ such that:
    \begin{enumerate}[(1)]

    \item
    $M\subseteq\cup_{\alpha}\phi_{\alpha}(U_{\alpha})$;

    \item
    For any pair $\alpha$, $\beta$ with $W=\phi_{\alpha}(U_{\alpha})\cap \phi_{\beta}(U_{\beta})\neq\emptyset$, the transition map
    $\phi_{\alpha}^{-1}\circ \phi_{\beta}|_{\phi_{\beta}^{-1}(W)}$ is a diffeomorphism from $\phi_\beta^{-1}(W)$
    onto $\phi_\alpha^{-1}(W)$, and
    $D[\phi_{\alpha}^{-1}\circ \phi_{\beta}]$ is  orthogonal matrix-valued everywhere;

    \item
    The family $\{(U_{\alpha},\phi_{\alpha})\}_\alpha$ is maximal with respect to the conditions (1) and (2).

    \end{enumerate}
For example,  lines,  circles,   planes,   spheres and the M\"{o}bius strip are all flat manifolds.
Note that a flat manifold becomes a metric space when it is equipped with the geodesic distance between points.

\begin{defn}
Let $M$ be a flat manifold of dimension $n$.
A function $f:M\to \mathbb{R}$ is \emph{locally} (resp.\ \emph{pointwise}) \emph{Lipschitz} if for all $p$ in $M$,
there is a chart $\phi_{p}:U\subseteq \mathbb{R}^n\to \phi_{p}(U)\subseteq M$ with $\phi_{p}(0)=p$
such that $f\circ \phi_{p}:U\subseteq \mathbb{R}^n\to \mathbb{R}$ is locally (resp.\ pointwise) Lipschitz at $0$.
Moreover, we define the \emph{local} (resp.\ \emph{pointwise}) \emph{Lipschitz constant} of $f$ at $p$ by
$\Lloc_p(f) = \Lloc_0(f\circ\phi_p)$ (resp.\ $\Lpt_p(f)=\Lpt_{0}(f\circ \phi_{p})$).
\end{defn}

\begin{lem}
Let $M$ be a flat manifold.
A function $f:M\to \mathbb{R}$ is locally \textup{(}resp.\ pointwise\textup{)} Lipschitz at $p$ with respect to a
chart $(\phi_{p}, U)$ is equivalent to the same property with respect to another chart $(\psi_{p},V)$ at $p$.  Indeed,
we have $\Lloc_0(f\circ\phi_p)=\Lloc_0(f\circ\psi_p)$ \textup{(}resp.\ $\Lpt_0(f\circ \phi_{p})=\Lpt_0(f\circ \psi_{p})$\textup{)}.
\end{lem}

\begin{proof}
Let $W=\phi_p(U)\cap \psi_p(V)$.
Observe that
$$
\Lpt_0(f\circ \phi_{p})=
\Lpt_0(f\circ\phi_p\mid_{\phi_p^{-1}(W)}) \quad\text{and}\quad \Lpt_0(f\circ\psi_p)=\Lpt_0(f\circ\psi_p\mid_{\psi_p^{-1}(W)}).
$$
It follows
\begin{align*}
\Lpt_0(f\circ \phi_{p})&=\Lpt_0(f\circ \psi_{p}\circ \psi_{p}^{-1} \circ \phi_{p}\mid_{\phi_p^{-1}(W)})\\
&\leq \Lpt_0(f\circ \psi_{p})\cdot \Lpt_0(\psi_{p}^{-1} \circ \phi_{p}\mid_{\phi_p^{-1}(W)})=\Lpt_0(f\circ \psi_{p})
\end{align*}
and
\begin{align*}
\Lpt_0(f\circ \psi_{p})&=\Lpt_0(f\circ \phi_{p}\circ \phi_{p}^{-1} \circ \psi_{p}\mid_{\psi_p^{-1}(W)})\\
&\leq \Lpt_0(f\circ \phi_{p})\cdot \Lpt_0(\phi_{p}^{-1} \circ \psi_{p}\mid_{\psi_p^{-1}(W)})=\Lpt_0(f\circ \phi_{p}).
\end{align*}
Hence $\Lpt_0(f\circ \phi_{p})=\Lpt_0(f\circ \psi_{p})$.

The case for the local Lipschitz constants is similar.
\end{proof}

\begin{thm}\label{thm:Flat}
Let $M$, $N$ be two $n$-dimensional flat manifolds. Let $\sigma :N\to M$
 such that
the  composition operator $Tf=f\circ \sigma$ defines a bijective linear map
 $T:\lip^{\pt}(M)\rightarrow\lip^{\pt}(N)$  satisfying that $\Lpt_x(Tf)=\Lpt_{\sigma(x)}(f)$
\textup{(}resp.\ $T:\lip^\loc(M)\rightarrow\lip^\loc(N)$  satisfying that $\Lloc_x(Tf)=\Lloc_{\sigma(x)}(f)$\textup{)} for all
$x$ in $N$.  Then $\sigma$ is
 a local isometry in the sense that for any point $p\in N$, and any
  chart $\phi: U\to M$ of $\sigma(p)$ and $\psi: V\to N$ of $p$
such that $\sigma(\psi(V))\subseteq \phi(U)$,
the induced map $\phi^{-1}\circ\sigma\circ \psi : V\to U$ is an isometry.
\end{thm}

\begin{proof}
Let $p$ be in $N$, equipped with charts $\phi:U\to M$ and $\psi:V\to N$ such that
$\psi(0)=p$, $\phi(0)=\sigma(p)$ and $\sigma(\psi(V))\subseteq \phi(U)$.
Note that both $U,V$ are open and connected in $\mathbb{R}^n$.
The composition map $T(g\circ\phi^{-1})\circ\psi =
g\circ(\phi^{-1}\circ\sigma\circ\psi)$
defines a bijection from $\lip^\pt(U)$ onto $\lip^\pt(V)$ preserving the pointwise Lipschitz constants.
It follows from Theorem \ref{thm:loca[0,1]n} that   $\phi^{-1}\circ\sigma\circ\psi$ extends to an isometry
from $\mathbb{R}^n$ onto $\mathbb{R}^n$.

The case for local Lipschitz functions is similar.
\end{proof}

\begin{exam}
Let $S^2$ be the unit sphere in $\mathbb{R}^3$.
Let $T:\lip^{\pt}(S^2)\rightarrow\lip^{\pt}(S^2)$ be a bijection such that $Tf=f\circ \sigma$,
and $\Lpt_p(Tf)=\Lpt_{\sigma(p)}(f)$ for all $p\in S^2$.
By Theorem \ref{thm:Flat},  $\sigma$ is a local isometry,
and thus a surjective isometry with respect to the geodesic metric on $S^2$.
\end{exam}

\begin{exam}
Let $0< r< R$ and
\begin{center}
$S^1\times S^1$ = $\{((R+r\cos \theta)\cos \phi,(R+r\cos \theta)\sin \phi,r\sin \theta)\in\mathbb{R}^3$ :
$0\leq \theta \leq2\pi$, $0\leq \phi \leq2\pi$ $\}$
\end{center}
be a $2$-dimensional torus.
Let $T:\lip^{\pt}(S^1\times S^1)\rightarrow\lip^{\pt}(S^1\times S^1)$ be a bijection such that $Tf=f\circ \sigma$ and $\Lpt_p(Tf)=\Lpt_{\sigma(p)}(f)$ for all $p\in S^1\times S^1$.
It follows from Theorem \ref{thm:Flat} that
 $\sigma$  is a local isometry, and thus it is
   a surjective isometry of $S^1\times S^1$ in the geodesic metric.
\end{exam}


\end{document}